\documentclass[11pt]{amsart}
\usepackage{amsxtra}
\usepackage{amssymb}
\usepackage{mathrsfs} 
\usepackage{enumerate}
\usepackage{array,color}
\usepackage{hyperref}
\usepackage{comment}
\theoremstyle{plain}
\newcommand{\cleqn}{\setcounter{equation}{0}}
\newcommand{\clth}{\setcounter{theorem}{0}}
\newcommand {\sectionnew}[1]{\section{#1}\cleqn\clth}
\newtheorem{theorem}{Theorem}[section]
\newtheorem{lemma}[theorem]{Lemma}
\newtheorem{definition-theorem}[theorem]{Definition-Theorem}
\newtheorem{proposition}[theorem]{Proposition}
\newtheorem{corollary}[theorem]{Corollary}
\newtheorem{definition}[theorem]{Definition}
\newtheorem{example}[theorem]{Example}

\newtheorem{remark}[theorem]{Remark}
\newtheorem{conjecture}[theorem]{Conjecture}
\newtheorem{notation}[theorem]{Notation}

\newcommand \bth[1] { \begin{theorem}\label{t#1} }
\newcommand \ble[1] { \begin{lemma}\label{l#1} }

\newcommand \bpr[1] { \begin{proposition}\label{p#1} }
\newcommand \bco[1] { \begin{corollary}\label{c#1} }
\newcommand \bde[1] { \begin{definition}\label{d#1}\rm }
\newcommand \bex[1] { \begin{example}\label{e#1}\rm }
\newcommand \bre[1] { \begin{remark}\label{r#1}\rm }
\newcommand \bcj[1] { \begin{conjecture}\label{j#1}\rm }

\newcommand \bnota[1] { \begin{notation}\label{n#1}\rm }
\renewcommand {\eth} { \end{theorem} }
\newcommand {\ele} { \end{lemma} }

\newcommand {\epr} { \end{proposition} }
\newcommand {\eco} { \end{corollary} }
\newcommand {\ede} { \end{definition} }
\newcommand {\eex} { \end{example} }
\newcommand {\ere} { \end{remark} }
\newcommand {\ecj} { \end{conjecture} }

\newcommand {\enota} { \end{notation} }
\newtheoremstyle{named}{}{}{\itshape}{}{\bfseries}{.}{.5em}{\thmnote{#3} #1}
\theoremstyle{named}

\newcommand{\vocab}[1]{\textbf{#1}}

\def \wt {\widetilde}

\def \p  {\mathfrak{p}}

\DeclareMathOperator \diag { {\mathrm{diag}} }

\DeclareMathOperator \tr { {\mathrm{tr}} }

\begin{document}
%%%%%%%%%%%%%%%%%%%%%%%%%%%%%%%%%%%%%%%%%%%%%%%%%%%%%%%%%%%%%%%%%%%%%%%%%%%
%%%%%%%%%%%%%%%%%%%%%%    Title    %%%%%%%%%%%%%%%%%%%%%%%%%%%%%%%%%%%%%%%%
\title[Pascal Matrix, Commuting Tridiagonal Operators and Fourier Algebras]
{The Pascal Matrix, Commuting Tridiagonal Operators and Fourier Algebras}
\author[W.~Riley Casper]{W.~Riley Casper}
\address{
Department of Mathematics \\
California State University Fullerton \\
Fullerton, CA 92831\\
U.S.A.
}
\email{wcasper@fullerton.edu}  
\author[Ignacio Zurri\'an]{Ignacio Zurri\'an}
\address{
Departamento de Matem\'{a}tica Aplicada II \\
Universidad de Sevilla \\
Seville, Spain
}
\email{ignacio.zurrian@fulbrightmail.org}
\date{}
\begin{abstract}
We consider the (symmetric) Pascal matrix, in its finite and infinite versions, and prove the existence of symmetric tridiagonal matrices commuting with it by giving explicit expressions for these commuting matrices. This is achieved by studying the associated Fourier algebra, which as a byproduct, allows us to show that all the linear relations of a certain general form for the entries of the Pascal matrix arise from only three basic relations.
We also show that pairs of eigenvectors of the tridiagonal matrix define a natural eigenbasis for the binomial transform.
Lastly, we show that the commuting tridiagonal matrices provide a numerically stable means of diagonalizing the Pascal matrix.
\end{abstract}
\maketitle
%\tableofcontents
\sectionnew{Introduction}
In random matrix theory and signal processing, certain natural integral operators exhibit the \emph{prolate spheroidal property} that they commute with a differential operator.
This fact leads to remarkable inroads to the study of the spectral properties of the integral operators \cite{S,SP,TW1,TW2}, resolving long-standing problems in both areas.
In this paper, we prove that the $N\times N$ (symmetric) Pascal matrix $T_N$ defined by
$$(T_N)_{jk} = \binom{j+k}{j},\quad  0\leq j,k < N,$$
satisfies a discrete analog of the prolate spheroidal property, in that it commutes with a Jacobi matrix, a discrete analog of a second-order differential operator.
As we let $N$ go to infinity, the matrix $T_N$ becomes the semi-infinite Pascal matrix 
$$
T = \left[\begin{array}{ccccc}
1 & 1 &  1  & \dots \\
1 &2 & 3 & \dots \\
 1  & 3 & 6 & \dots \\
\vdots & \vdots & \vdots & \ddots\\
\end{array}\right],
$$
which we also prove has the discrete prolate spheroidal property.
This is stated explicitly in the following main theorem.

\smallskip

\noindent
{\bf{Main Theorem.}} {\em{
The semi-infinite Pascal matrix $T$ commutes with the two semi-infinite Jacobi matrices
$$
J = \left[\begin{array}{ccccc}
b_0 & a_1 &  0  & \dots \\
a_1 & b_1 & a_2 & \dots \\
 0  & a_2 & b_2 & \dots \\
\vdots & \vdots & \vdots & \ddots\\
\end{array}\right]
\quad\text{and}\quad
\wt J = \left[\begin{array}{cccc}
\beta_0 & \alpha_1 &  0  & \dots\\
\alpha_1 & \beta_1 & \alpha_2 & \dots\\
 0  & \alpha_2 & \beta_2 & \dots\\
\vdots & \vdots & \vdots & \ddots\\
\end{array}\right]
$$
where here
$$a_n = n,\quad b_n = -n,\quad \alpha_n = n^3\quad\text{and}\quad \beta_n = -2n^3-3n^2-2n.$$
Moreover, the $N\times N$ Pascal matrix $T_N$ commutes with the $N\times N$ Jacobi matrix $J_N$ whose entries are given by
$$(J_N)_{jk} = (N^2J - \wt J)_{jk},\quad 0\leq j,k < N.$$

}}
\medskip
The finite matrix $J_N$ is a Jacobi matrix, so its spectrum is simple.
This implies that the eigenvectors of $J_N$ are also eigenvectors of $T_N$.
Moreover, we prove that the binomial transform acts as an involution interchanging the eigenspaces of $J_N$, so the eigenvectors of the  binomial involution have simple expressions in terms of the eigenvectors of $J_N$.

More generally, we consider the $LU$-factorization $T = \Psi\Psi^\intercal$ for $\Psi$ the semi-infinite lower triangular matrix with entries 
$$\Psi_{jk} = \binom{j}{k}.$$
We classify all pairs $A$ and $B$ of semi-infinite matrices with finite bandwidth satisfying
$$A\Psi = \Psi B.$$
These pairs form an algebra which, as we prove in Theorem \ref{tthm:Fourier}, is generated by three simple elements.
As an interesting consequence, this says that \emph{all} linear relations for binomial coefficients of the form
$$\sum_{j=-m}^m a_j(x)\binom{x+j}{y} = \sum_{k=-n}^n b_k(y)\binom{x}{y+k},$$
for functions $a_j(x)$ and $b_k(y)$ on the natural numbers,
arise from the three basic relations
$$\binom{x+1}{y} = \binom{x}{y}+\binom{x}{y-1}
,\quad \binom{x}{y+1} = \frac{x-y}{y+1}\binom{x}{y}\quad\text{and}\quad\binom{x-1}{y} = \frac{x-y}{x}\binom{x}{y}.$$
Exploration of the relationship between the matrices $A$ and $B$ also results in interesting new combinatorial identities, such as Equation \eqref{fubar}.

Another nice application of the existence of $J_N$ is for numerical methods of identifying the eigenvectors and eigenvalues of $T_N$.
This is challenging from the point of view of numerical stability since the spectrum of $T_N$ has simultaneously a wide range and a concentration of values near zero.
In contrast, the spectrum of $J_N$ is simple and nicely spread out, so standard algorithms can reliably obtain the spectral data.
Alternatively, one can leverage the spectral theory of Jacobi matrices to express the eigenfunctions in terms of the sequence of orthogonal polynomials associated with $J_N$.
Moreover, the simplicity of the spectrum implies that the eigenfunctions of $J_N$ are also eigenfunctions of $T_N$.
Thus the existence of $J_N$ provides a fast and numerically stable way of obtaining an eigenbasis for $T_N$.
Numerical diagonalization of $T_N$ is explored in the last section of the paper.

It is worth mentioning that one can take a very pedestrian approach to determining the matrix $J_N$ commuting with $T_N$.
Specifically, one can certainly find it explicitly for small values of $N$, and then by detecting a pattern, one could guess a general expression for $J_N$ for arbitrary $N$.
Then with some significant calculations, one can prove the commutation formula in full generality. However, this is not exactly satisfactory since it doesn't explain \emph{why} $J_N$ exists.
In particular, it neglects the rich algebraic structure underlying the semi-infinite matrix $T$.
%While the fact that $T_N$ and $J_N$ commute can be obtained directly, this is not exactly satisfactory since it doesn't explain \emph{why} $J_N$ exists.

We show that the
%The 
existence of $J_N$ stems from an intrinsic connection between the Pascal matrix $T_N$ and the notion of bispectral functions.
Bispectral functions were originally introduced in \cite{DG}, and the motivation for their study came from time and band limiting to explain the existence of operators with the prolate spheroidal property.
Recently, a general framework for creating integral operators with the prolate spheroidal property from bispectral functions was established \cite{CY1,CGYZ1,CGYZ2}, by introducing the notion of Fourier algebras, see also \cite{CGYZ3,CGYZ4}.
We establish the existence of $J_N$ here using this same framework.
This also serves as a remarkable demonstration of the utility of Fourier algebras in the very concrete situation of matrix linear algebra.

One very useful application of the eigenfunctions of $J_N$ is that they provide a canonical basis for the binomial transform on $\mathbb{C}^N$.
The eigenvalues of $J_N$ have the symmetry property that if $\vec v$ is an eigenvector of $J_N$ with eigenvalue $\lambda$, then its binomial transform
$$w_j = \sum_{k=0}^{N-1}\binom{j}{k}(-1)^kv_k$$
is an eigenvector of $J_N$ with eigenvalue $N^2-1-\lambda$.
Then the simple linear combination
$$\vec v \pm \vec w$$
defines an eigenvector of the binomial transform with eigenvalue $\pm 1$.
For another recent paper exploring the connection between semi-infinite triangular Pascal matrices and eigenfunctions of the binomial transform emphasizing applications to open questions posed by Shapiro \cite{shapiro}, see \cite{tsatsomeros}.

\smallskip

\noindent
{\bf{Notation.}} {
Throughout the paper, we will adopt the following notation:
\begin{itemize}
\item $\mathbb{N}$ will denote the nonnegative integers $\{0,1,2,\dots\}$,
\item matrix indices $A_{jk}$ will start from $0$.
\end{itemize}
}

\sectionnew{Bispectrality and Fourier algebras}
\subsection{Bispectral functions}
The original notion of a bispectral function was a locally analytic function of two variables $\psi(x,y)$, for which there exists differential operators $P(x,\partial_x)$ and $Q(y,\partial y)$ such that $\psi$ is \emph{simultaneously} a family of eigenfunctions of $P(x,\partial_x)$ and a family of eigenfunctions for $Q(y,\partial_y)$.
The simplest example of this is the exponential function $\psi(x,y) = e^{xy}$ which satisfies
$$\frac{\partial}{\partial x} \psi(x,y) = \psi(x,y)y,\quad\text{and}\quad \frac{\partial}{\partial y} \psi(x,y) = \psi(x,y)x.$$
Airy and Bessel functions give even more interesting examples of bispectral functions in this classical sense.

Since then, the notion of bispectrality has been extended in several directions, such as difference operators and matrix-valued operators, adding connections to other areas of mathematics such as orthogonal polynomials.
Now, the original version of bispectral functions is referred to as continuous-continuous bispectral functions.
In this paper, we are interested in discrete-discrete bispectral functions: functions which are simultaneously families of eigenfunctions of difference operators in each variable.

In this work, a \vocab{difference operator} will refer to an expression of the form
$$L(x,\delta_x) = \sum_{k=0}^n a_k(x)\delta_x^k + \sum_{k=1}^{n} a_{-k}(x)(\delta_x^k)^*,$$
for some integer $n\geq 0$ and functions $a_k: \mathbb{N}\rightarrow\mathbb{R}$.
This operator has a natural action on the vector space of $\mathbb{N}$-valued functions on $\mathbb{N}$ by
$$L(x,\delta_x)\cdot f(x) = \sum_{k=-n}^n a_k(x)f(x+k),$$
where we adopt the convention that $f(x)$ is zero for negative values of $x$.
Note: the algebra of difference operators comes with a natural anti-involution $*$, which transposes $\delta$ and $\delta^*$ and sends a function $f(x)$ to itself.

Every difference operator has a natural representation as a semi-infinite matrix in an obvious way.
Notationally, we will distinguish between a shift operator $L$ and its matrix representation, which we denote by $\pi(L)$.
In general
$$\pi\left(\sum_{k=0}^n a_k(x)\delta_x^k + \sum_{k=1}^n a_{-k}(x)(\delta_x^k)^*\right) =
\left[\begin{array}{cccc}
a_0(0) & a_1(0) & a_2(0) & \dots\\
a_{-1}(1) & a_0(1) & a_1(1) & \dots\\
a_{-2}(2) & a_{-1}(2) & a_0(2) & \dots\\
\vdots & \vdots & \vdots & \ddots\\
\end{array}\right].$$
The expression is defined the same when $x$ is replaced by $y$.
Functions on $\mathbb{N}$ naturally correspond with semi-infinite vectors, and under this correspondence, the action of $L$ on functions corresponds to the action of $\pi(L)$ on vectors.
Note also that $\pi(L^*) = \pi(L)^\intercal$ for all difference operators $L$.

With our notation for difference operators established, we can formally define bispectrality in the context of this paper.
\bde{def:bispectrality}
A function $\psi: \mathbb{N}\times\mathbb{N}\rightarrow\mathbb{R}$ is called \vocab{bispectral} if there exist difference operators $L(x,\delta_x)$ and $R(y,\delta_y)$ and complex-valued functions $f(x)$ and $g(y)$ on $\mathbb{N}$ with the property that
$$L(x,\delta_x)\cdot \psi(x,y) = g(y)\psi(x,y)\quad\text{and}\quad R(y,\delta_y)\cdot\psi(x,y) = f(x)\psi(x,y).$$
The associated operators $L(x,\delta_x)$ and $R(y,\delta_y)$ are called \vocab{bispectral operators}.
\ede

The exponential function again provides a simple example of a bispectral function.
In fact,
$$\delta_x e^{xy} = ye^{xy}\quad\text{and}\quad\delta_y e^{xy} = xe^{xy}.$$
However, in this paper we are not interested in studying arbitrary discrete-discrete bispectral functions.
With our application of studying Pascal's matrix in mind, the most important example of a bispectral function will be the binomial function.

\bth{thm:bispectral}
The function $\psi: \mathbb{N}\times\mathbb{N}\rightarrow \mathbb{R}$ defined by
$$\psi(x,y) = \binom{x}{y}$$
is bispectral and satisfies
$$((y+1)\delta_y+y)\psi(x,y)=x\psi(x,y)\quad\text{and}\quad(x-x\delta^*_x)\psi(x,y) = y\psi(x,y).$$
\eth
\begin{proof}
This is a restatement of the two standard binomial identities
$$\binom{x}{y+1} = \frac{x-y}{y+1}\binom{x}{y}\quad\text{and}\quad\binom{x-1}{y} = \frac{x-y}{x}\binom{x}{y}.$$
\end{proof}

\begin{remark}
For the remainder of the paper, unless specified otherwise $\psi(x,y)$ will refer to the binomial function.
\end{remark}

The bispectral function $\psi(x,y)$ can itself be encoded as a semi-infinite lower triangular matrix 
$$
\Psi = \left[\begin{array}{ccccc}
1 & 0 & 0  &0  & \dots \\
1 &1 & 0 & 0  &\dots \\
 1  & 2 & 1 & 0  &\dots \\
 1  & 3 & 3 & 1  &\dots \\
\vdots & \vdots & \vdots & \vdots&\ddots\\
\end{array}\right],
$$ with the entries defined by
$$\Psi_{j,k}=\psi(j,k).$$
The matrix $\Psi$ itself is sometimes referred to as the lower triangular Pascal matrix.

With $\psi$ encoded as a matrix, bispectrality may also be written as a matrix identity.
In particular
$$L(x,\partial_x)\cdot\psi(x,y) = R(y,\partial_y)\cdot\psi(x,y)\ \Leftrightarrow\ \pi(L)\Psi = \Psi \pi(R)^\intercal.$$

\subsection{Fourier algebras}
One of the key developments in the push toward classifying large families of bispectral functions was to consider the algebra of \emph{all} bispectral operators for which a bispectral function $\psi(x,y)$ is a family of eigenfunctions.
This leads to the definition of the left and right \vocab{bispectral algebras} 
$$\mathcal B_x(\psi) = \{L(x,\delta_x) : L(x,\delta_x)\cdot \psi(x,y) = g(y)\psi(x,y),  \text{ for some function } g(y) \},$$
$$\mathcal B_y(\psi) = \{R(y,\delta_y) : R(y,\delta_y)\cdot \psi(x,y) = f(x)\psi(x,y),  \text{ for some function } f(x) \}.$$
In the classical setting of continuous-continuous bispectral functions, the direct study of the algebraic structure of the algebra $\mathcal B_x(\psi)$ allowed Wilson to classify all bispectral functions for which $\mathcal B_x(\psi)$ contained multiple operators of relatively prime degree \cite{Wilson}.
This connects the study of bispectrality to other fields of mathematics, including algebraic geometry and integrable systems.

Wilson's insight was extended even further, starting with \cite{BHY1,GHY}.
One can enlarge the bispectral algebras, by allowing the eigenvalues to be ``operator-valued", too. 
This leads to the definition of the \vocab{left and right Fourier algebras}
$$\mathcal F_x(\psi) =\{L : L(x,\delta_x)\cdot \psi(x,y) =R(y,\delta_y)\cdot \psi(x,y),  \text{ for some difference operator } R \},$$
$$\mathcal F_y(\psi) =\{R : L(x,\delta_x)\cdot \psi(x,y) =R(y,\delta_y)\cdot \psi(x,y),  \text{ for some difference operator } L \}.$$
Exploration of these more general Fourier algebras has resulted in recent progress on some long-standing problems in bispectrality and special functions, such as \cite{CY1,CY2}.

Under certain nondegeneracy conditions, the left and right Fourier algebras turn out to be anti-isomorphic, via a map called the \vocab{generalized Fourier map} $b_\psi$, which is defined by
\begin{equation}\label{eqn:fourier map}
b_\psi: L(x,\delta_x)\mapsto R(y,\delta_y)\quad\text{if and only if}\quad L(x,\delta_x)\cdot \psi(x,y) =R(y,\delta_y)\cdot \psi(x,y).
\end{equation}

The naming convention again arises directly from the classical setting.
For the exponential function $e^{xy}$, the left and right Fourier algebras are both the algebras of differential operators with polynomial coefficients, and the generalized Fourier map is the same as the Fourier transform acting on these operator algebras.

Understanding the structure of the Fourier algebra and its interplay with the bispectral function has resulted in large leaps in progress in multiple areas of mathematics, including the study of matrix orthogonal polynomials satisfying differential equations \cite{CY2} and the connection between bispectrality, integrable systems, and the prolate spheroidal property \cite{CY1,CGYZ1,CGYZ2}.

To start out, we wish to show that the generalized Fourier map is well-defined and an isomorphism for the binomial bispectral function $\psi(x,y)$.
The bispectral map can be expressed in terms of matrix representations as
$$\pi(L)\Psi = \Psi \pi(b_\psi(L))^\intercal,$$
for all $L\in \mathcal F_x(\psi)$.
For this reason, $b_\psi$ being well-defined comes down to the fact that $\Psi$ is invertible.

\begin{lemma}
Let $\psi(x,y) = \binom{x}{y}$.
Then the generalized Fourier map $b_\psi: \mathcal F_x(\psi)\rightarrow F_y(\psi)$ defined by \eqref{eqn:fourier map} is well-defined and an anti-isomorphism.
\end{lemma}
\begin{proof}
To show that $b_\psi$ is well-defined, it suffices to show that for any $L(x,\delta_x)\in\mathcal F_x(\psi)$, there is a unique $R(y,\delta_y)\in\mathcal F_y(\psi)$ with $L\cdot\psi = R\cdot\psi$.
By definition of the Fourier algebra, there is at least one.
Now suppose that $\wt R(y,\delta_y)$ also satisfies $L\cdot\psi = R\cdot\psi$.
Then $(R-\wt R)\cdot\psi = 0$, which in terms of matrices implies
$$(\pi(R)-\pi(\wt R))\Psi = 0.$$
The matrix $\Psi$ is invertible, so this implies $\pi(R) = \pi(\wt R)$ and therefore $R = \wt R$.
Thus $b_\psi$ is well-defined.

The map $b_\psi$ is surjective, again directly from the definition of the Fourier algebras.
Moreover, the same kind of argument in the previous paragraph implies that for any $R(y,\partial_y)\in\mathcal F_y(\psi)$ there exists a unique $L(x,\partial_x)\in \mathcal F_x(\psi)$ with $L\cdot\psi = \psi\cdot R$.  Thus $b_\psi$ is injective.

It is easy to see from the definition that $b_\psi$ is a linear transformation.
Furthermore, for any $L_1,L_2\in\mathcal F_x(\psi)$, we see
$$(L_1L_2)\cdot\psi = L_1 b_\psi(L_2)\cdot\psi = b_\psi(L_2)L_1\cdot\psi = b_\psi(L_2)b_\psi(L_1)\cdot\psi.$$
Thus $b_\psi(L_1L_2) = b_\psi(L_2)b_\psi(L_1)$.
Hence $b_\psi$ is an anti-isomorphism.
\end{proof}

Next, we wish to identify the left and right Fourier algebras in our setting along with the generalized Fourier map.
The classic defining characteristic of Pascal's triangle
$$\binom{x+1}{y} = \binom{x}{y}+\binom{x}{y-1}$$
can be expressed in our language as
$$\delta_x\cdot\psi(x,y) = (1 + \delta_y^*)\cdot\psi(x,y).$$
Put together with the previous two identities from Equation \eqref{Fourier-gen}, we can see that the generalized Fourier map will satisfy
\begin{equation}\label{Fourier-gen}
    x\mapsto y + (y+1)\delta_y,\quad \delta_x\mapsto\delta_y^*+1,\quad \text{and}\quad  x\delta_x^*\mapsto(y+1)\delta_y.
\end{equation}
In particular, this implies that $F_x(\psi)$ contains the algebra $\mathbb R[x,\delta_x,x\delta_x^*]$ of all difference operators generated by $x$, $\delta_x$, and $x\delta_x^*$.  Likewise, $\mathcal F_y(\psi)$ contains $\mathbb R[y,\delta_y^*,(y+1)\delta_y]$.

Even more interestingly, we can show that this is the entire Fourier algebra!
This means that Equation \eqref{Fourier-gen} defines the entire generalized Fourier map.
Moreover it says somehow that all the identities having to do with finite linear sums of binomial coefficients somehow are generated from three basic identities.
However, proving that the Fourier algebra is generated by the three elements we specified is challenging.
In particular, it relies on some results about the Hilbert space $\ell^2(\mathbb{N})$ with the signed measure $\mu(y) = (-1)^y$.

The binomial coefficients satisfy the orthogonality identity
\begin{equation}\label{orthogonal}\sum_{y\in\mathbb N}(-1)^y\binom{x}{y}\binom{y}{x'} = \left\lbrace\begin{array}{cc}
1, & x=x',\\
0, & \text{otherwise}.
\end{array}\right.\end{equation}
To see how to use this, suppose we have two difference operators
$$L(x,\delta_x) = \sum_{j=0}^m a_j(x)\delta_x^j + \sum_{j=1}^{m} a_{-j}(x)(\delta_x^j)^*,\ \ \text{and}\ \ R(y,\delta_y) = \sum_{k=0}^n b_k(y)\delta_u^k + \sum_{k=1}^{n} b_{-k}(y)(\delta_y^k)^*,$$
with $L\cdot\psi = R\cdot\psi$.
Then
$$\sum_{j=-m}^m a_j(x)\binom{x+j}{y} = \sum_{k=-n}^n b_k(y)\binom{x}{y+k},$$
and by the orthogonality relations, we obtain the two equations
\begin{align}\label{kung foo}
a_{\ell}(x) &= \sum_{k=-n}^n\sum_{y=0}^\infty (-1)^y b_k(y)\binom{y}{x+\ell} \binom{x}{y+k},\\
b_{\ell}(y) &= \sum_{j=-m}^m \sum_{x=0}^\infty (-1)^x a_j(x)\binom{y+\ell}{x} \binom{x+j}{y}.
\end{align}
This results in some very interesting combinatorial identities for alternating sums of binomial coefficients.
For example, from the generalized Fourier map we know that
\begin{align*}
b_\psi(x^n)
 & =  (y + (y+1)\delta_y)^n\\
 & = y^n + \sum_{i=0}^{n-1}y^i(y+1)^{n-i} \delta_y + \sum_{0\leq i\leq j\leq n-2} y^i(y+1)^{j-i+1}(y+2)^{n-1-j}\delta_y^2 + \dots\\
 & = \sum_{\ell=0}^n \sum_{0\leq i_1\leq i_2\leq \dots\leq i_\ell\leq n-\ell} y^{i_1}(y+1)^{i_2-i_1+1}(y+2)^{i_3-i_2+1}\dots (y+\ell)^{n-\ell-i_\ell+1}\delta_y^{\ell}.
\end{align*}
Examining the coefficients of both difference operators, Equation \eqref{kung foo} gives us the interesting combinatorial identity
\begin{align}\label{fubar}
\sum_{x=0}^\infty (-1)^x x^n\binom{y+\ell}{x} \binom{x}{y}
 & = \sum_{0\leq i_1\leq \dots\leq i_\ell\leq n-\ell} \left(\prod_{j=0}^{\ell}(y+j)^{i_{j+1}-i_j+1}\right),
\end{align}
where here $i_0 = 1$ and $i_{\ell+1} = n-\ell$.

With this in mind, we have the following theorem.
\bth{thm:Fourier}\label{thm:Fourier}
The left and right Fourier algebras of $\psi(x,y) = \binom{x}{y}$ are given by
$$\mathcal F_y(\psi) = \mathbb R[y,\delta_y^*,(y+1)\delta_y],\quad \text{and}\quad \mathcal F_x(\psi) = \mathbb R[x,\delta_x,x\delta_x^*],$$
and the generalized Fourier map is the algebra anti-isomorphism $b_\psi: \mathcal F_x(\psi)\rightarrow\mathcal F_y(\psi)$ induced by Equation \eqref{Fourier-gen}.
\eth
\begin{proof}
After the above discussion, it suffices to show $\mathcal F_x(\psi) = \mathbb R[x,\delta_x,x\delta_x^*]$.
Then since $b_\psi$ is an isomorphism, the rest of the statement follows automatically.
Furthermore, it is easy to see that $\mathbb R[x,\delta_x,x\delta_x^*]$ consists of all difference operators of the form
$$L(x,\delta_x) = \sum_{j=0}^m a_j(x)\delta_x^j + \sum_{j=1}^{m} a_{-j}(x)(\delta_x^j)^*$$
for some integer $m$, with the property that each $a_j(x)$ is a polynomial in the variable $x$, and that for all $n>0$
$$a_{-n}(k) = 0,\quad\text{for all}\ 0 \leq k < n.$$
Thus we need only show that this characterizes the Fourier algebra.

Consider two difference operators
$$L(x,\delta_x) = \sum_{j=0}^m a_j(x)\delta_x^j + \sum_{j=1}^{m} a_{-j}(x)(\delta_x^j)^*,\ \ \text{and}\ \ R(y,\delta_y) = \sum_{k=0}^n b_k(y)\delta_u^k + \sum_{k=1}^{n} b_{-k}(y)(\delta_y^k)^*$$
with $L\cdot\psi = R\cdot\psi$.
It is clear from Equation \eqref{kung foo} that $$a_{-n}(k) = 0,\quad\text{for all}\ 0 \leq k < n.$$
Therefore to complete the proof, we need to show that $a_j(x)$ is a polynomial for each $j$.

Since
$$\sum_{j=-m}^m a_j(x)\binom{x+j}{y} = \sum_{k=-n}^n b_k(y)\binom{x}{y+k},$$
we see that
$$p(x,y) := \sum_{j=-m}^m  a_j(x) \binom{x+j}{y}$$
is a polynomial in $x$ of degree at most $y+n$ for all $y$.
This means that we have a linear system of equations
$$\left[\begin{array}{cccc}
\binom{x-m}{0} & \binom{x-m+1}{0} & \dots & \binom{x+m}{0}\\
\binom{x-m}{1} & \binom{x-m+1}{1} & \dots & \binom{x+m}{1}\\
\vdots & \vdots & \ddots & \vdots\\
\binom{x-m}{2m} & \binom{x-m+1}{2m} & \dots & \binom{x+m}{2m}\\
\end{array}\right]
\left[\begin{array}{c}
a_{-m}(x) \\ a_{1-m}(x)\\\vdots\\a_m(x)
\end{array}\right]
= \left[\begin{array}{c}
p(x,0) \\ p(x,1)\\\vdots\\p(x,2m)
\end{array}\right].$$
The entries in the matrix on the left-hand side of the display above are polynomials in $x$.  Furthermore, by \cite[Theorem 2.1]{Helou} the determinant of this matrix is identically $1$.
Therefore Cramer's rule implies that $a_j(x)$ is a polynomial in $x$ for all $j$.
\end{proof}

\sectionnew{Matrices commuting with Pascal's matrix}
The semi-infinite Pascal matrix $T$ has a natural expression in terms of $\Psi$.  Specifically the well-known binomial identity \cite[Equation 3.13]{spiegel}
$$\sum_{i=0}^{\min(j,k)} \binom{j}{i}\binom{k}{i} = \binom{j+k}{k}$$
implies the following LU factorization of $T$ %\cite{Brawer}
\begin{align}
T = \Psi\Psi^\intercal.
\end{align}

As a consequence of this factorization, we can encode the condition that the matrix representation of a particular operator in the Fourier algebra commutes with $T$ in terms of a condition of its interplay with the generalized Fourier map.
\bth{thm:commuting}
The matrix representation $\pi(L)$ of an operator $L\in \mathcal F_x(\psi)$ commutes with the semi-infinite Pascal matrix $T$ if and only if $L^*\in \mathcal F_x(\psi)$ and
$$b_\psi(L^*) = b_\psi(L)^*.$$
\eth
\begin{proof}
Suppose $L\in\mathcal F_x(\psi)$.
Then $\pi(L)$ commutes with $T$ if and only if
\begin{align*}
\pi(L)\Psi\Psi^\intercal&=\Psi\Psi^\intercal\pi(L)\\
\Leftrightarrow\Psi\pi(b_\psi(L))\Psi^\intercal&=\Psi\Psi^\intercal\pi(L).    
\end{align*}
Furthermore, since $\Psi$ is a lower triangular matrix with only ones in the diagonal it is invertible.
Therefore the above equation reduces to
$$\pi(b_\psi(L))\Psi^\intercal=\Psi^\intercal\pi(L).$$
Taking the transpose and using the fact that $\pi(R^*) = \pi(R)^\intercal$ for all difference operators $R$, this simplifies to
$$\Psi\pi(b_\psi(L)^*)=\pi(L^*)\Psi.$$
This is equivalent to $L^*$ belonging to $\mathcal F_x(\psi)$ with $b_\psi(L^*) = b_\psi(L)^*$.
\end{proof}

\subsection{Proof of the Main Theorem}\label{proofMT}
In this section, we prove the main theorem, whose statement is broken up into Theorems \ref{T1} and \ref{T2} below.
\begin{theorem}\label{T1}
The semi-infinite Pascal matrix $T$ commutes with the two semi-infinite Jacobi matrices
$$
J = \left[\begin{array}{ccccc}
b_0 & a_1 &  0  & \dots \\
a_1 & b_1 & a_2 & \dots \\
 0  & a_2 & b_2 & \dots \\
\vdots & \vdots & \vdots & \ddots\\
\end{array}\right]
\quad\text{and}\quad
\wt J = \left[\begin{array}{cccc}
\beta_0 & \alpha_1 &  0  & \dots\\
\alpha_1 & \beta_1 & \alpha_2 & \dots\\
 0  & \alpha_2 & \beta_2 & \dots\\
\vdots & \vdots & \vdots & \ddots\\
\end{array}\right]
$$
where here
$$a_n = n,\quad b_n = -n,\quad \alpha_n = n^3,\quad\text{and}\quad \beta_n = -2n^3-3n^2-2n.$$ 
\end{theorem}
\begin{proof}
Firstly, let us consider the operators for which the matrix representations are $J$ and $\wt J$. Namely, for 
\begin{align*}
    L(x,\delta_x)&=\delta_x x-x+x \delta_x^*,\\
    \wt L(x,\delta_x)&=\delta_x x^3-(2x^3+3x^2+2x)+x^3 \delta_x^*,
\end{align*}
we have $\pi(L)=J$ and $\pi(\wt L)=\wt J$.
It is immediate to see that  $L^*=L$ and $\wt L^*=\wt L$. On the other hand, by using Theorem \ref{thm:Fourier} we have that $L$ and $\wt L$ belong to $\mathcal F_x(\psi)$. Moreover, from \eqref{Fourier-gen}, we have 
\begin{align*}
%b_\psi(L)&&=b_\psi(\delta_x x-x+x %\delta_x^*)=b_\psi(\delta_x x)-%y=b_\psi(x)b_\psi(\delta_x)-y
%\\
%&=(y + (y+1)\delta_y)(\delta_y^*+1)-y
%&=y\delta_y^*+(y+1)\delta_y\delta_y^*+%(y+1)\delta_y
b_\psi(L)
&=b_\psi(\delta_x x-x+x \delta_x^*)\\
&=(y + (y+1)\delta_y)(\delta_y^*+1)-y\\
&=y\delta_y^*+ y+1+\delta_y y,
\end{align*}
showing that $b_\psi(L)=b_\psi(L)^*$. Similarly, we also have
\begin{align*}
b_\psi(\wt L)=&b_\psi(\delta_x x^3+x^3 \delta_x^*-(2x^3+3x^2+2x))
\\
=&(y + (y+1)\delta_y)^3(\delta_y^*+1)  
+ ((y+1)\delta_y)(y + (y+1)\delta_y)^2
\\
&-(2(y + (y+1)\delta_y)^3+3(y + (y+1)\delta_y)^2+2(y + (y+1)\delta_y))
%\\
%=&(y + (y+1)\delta_y)^3(\delta_y^*)  
%+ (-y)(y + (y+1)\delta_y)^2
%\\
%&-(3(y + (y+1)\delta_y)^2+2(y + %(y+1)\delta_y))
%\\
%=&(y + (y+1)\delta_y)^2 (y\delta_y^* + (y+1))
%\\
%&-((y+3)(y + (y+1)\delta_y)^2+2(y + (y+1)\delta_y))
%\\
%=&(y^2+(1+2y)(1+y)\delta_y+(1+y)(2+y)\delta_y^2)(y\delta_y^* + (y+1))
%\\
%&-((y+3)(y^2+(1+2y)(1+y)\delta_y+(1+y)(2+y)\delta_y^2)+2(y + (y+1)\delta_y))
%\\
%=&(y^2+(1+2y)(1+y)\delta_y)(y\delta_y^* + (y+1))+(1+y)(2+y)^2\delta_y
%\\
%&-((y+3)(y^2+(1+2y)(1+y)\delta_y)+2(y + (y+1)\delta_y))
%\\
%=&y^3\delta_y^*+(1+2y)(1+y)^2+(3+3y)(1+y)(2+y)\delta_y+y^2(y+1)
%\\
%&-(y+3)y^2-2y-\left((y+3)(1+2y)(1+y)+2(y+1)\right)\delta_y
%\\
%=&y^3\delta_y^*+(1+2y)(1+y)^2+y^2(y+1)-(y+3)y^2-2y+(1+y)^3\delta_x,
\\
=&y^3\delta_y^*+1 + 2 y + 3 y^2 + 2 y^3+\delta_y y^3.
\end{align*}
From this, it is manifest that $b_\psi(\wt L)=b_\psi(\wt L)^*$.

Finally, we obtain the desired result by applying Theorem \ref{tthm:commuting} to $L$ and $\wt L$.
\end{proof}
\begin{theorem}\label{T2}
The $N\times N$ Pascal matrix $T_N$ commutes with the $N\times N$ Jacobi matrix $J_N$ whose entries are given by
$$(J_N)_{jk} = (N^2J - \wt J)_{jk},\quad 0\leq j,k < N.$$
\end{theorem}
\begin{proof}
If we define the semi-infinite matrix $J=N^2J - \wt J$, it is clear that it commutes with $T$ since $J$ and $\wt J$ commute with $T$. On the other hand, by construction, $J$ is of the form
$$J=
\left[\begin{array}{c|c}
J_N&\bf 0\\\hline
\bf 0 & *
\end{array}\right],
$$
for $J_N$ a $N\times N$ matrix. 
Also, recall that $T_N$ is just the left-upper $N\times N$ sub-matrix of $T$, i.e., 
$$T=
\left[\begin{array}{c|c}
T_N&\bf *\\\hline
\bf * & \bf*
\end{array}\right].
$$
Then, it follows that $J_N$ commutes with $T_N$.
\end{proof}

% We will let $\pi_N(L)$ denote the left upper $N\times N$ sub-matrix of $\pi(L)$. 
% We denote by $I_N$ the identity matrix of size $N$ and $\chi_N$ the semi-infinite matrix given by 
% $$\chi_n=
% \left(\begin{array}{c|c}
% I_N&\bf 0\\\hline
% \bf 0 & \bf0
% \end{array}\right).
% $$
% Here, by abuse of notation, $\bf0$ means that there are only zeros in that (semi-infinite) block.
% If we have a semi-infinite matrix $J$ such that it commutes with $T$ and it is of the form
% $$J=
% \left(\begin{array}{ccc|cc}
% J_{00}&\dots&J_{0,N-1}&0&\cdots\\
% \vdots&\ddots&\vdots&\vdots&\cdots\\
% J_{N-1,0}&\dots&J_{N-1,N-1}&0&\cdots\\\hline
% 0&\dots&0&J_{NN}&\cdots\\
% \vdots&\vdots&\vdots&\vdots&\ddots\\
% \end{array}\right)
% $$
% $$J=
% \left(\begin{array}{c|c}
% J_N&\bf 0\\\hline
% \bf 0 & *
% \end{array}\right)
% $$

\subsection{A dual commuting pair}
Let $\Psi_N$ be the upper left $N\times N$ submatrix of the semi-infinite matrix $\Psi$.
Theorem \ref{T2} gives a tridiagonal matrix commuting with $T_N=\Psi_N\Psi_N^\intercal$, in this section we obtain a tridiagonal matrix that commutes with the symmetric matrix ``dual" to $T_N$, namely
$$S_N := \Psi_N^\intercal\Psi_N.$$
The entries of $S_N$ are given by
$$(S_N)_{jk} = \sum_{i=0}^{N-1}\binom{i}{j}\binom{i}{k}.$$

As an abuse of notation, we let $b_\psi(J_N)$ denote the upper $N\times N$ submatrix of the semi-infinite matrix associated with the difference operator
$$b_\psi(J)=N^2b_\psi(L)-b_\psi(\wt L).$$%Comment: I added the left-hand side of this display. Also: instead of b_\psi(J_N) we could use {b_\psi(J)}_N
Our expression for the value of the generalized Fourier map in the previous section shows that
\begin{equation}\label{bJ}
    b_\psi(J_N) = -\Lambda_N J_N \Lambda_N + (N^2-1)I_N,
\end{equation}
where here $I_N$ the $N\times N$ identity matrix and $\Lambda_N$ the diagonal matrix
$$\Lambda_N = \diag(1,-1,1,\dots,(-1)^{N-1}).$$

Since $\Psi$ is a lower triangular matrix, the truncated matrix $T_N$ is related to the truncation $\Psi_N$ of $\Psi$ by
$$T_N = (\Psi\Psi^\intercal)_N = \Psi_N\Psi_N^\intercal.$$
This automatically provides us with a matrix commuting with $\Psi_N^T\Psi_N$.
\begin{proposition}
The matrix $b_\psi(J_N)$ commutes with $\Psi^\intercal\Psi$.
\end{proposition}
\begin{proof}
The fact that $J_N$ commutes with $T_N$ comes from the relationship
$$J_N\Psi_N\Psi_N^\intercal = \Psi_N b_\psi(J_N)\Psi_N^\intercal = \Psi_N\Psi_N^\intercal J_N,$$
where here we are using the fact that $b_\psi(J_N)^\intercal = b_\psi(J_N)$.
Therefore
$$b_\psi(J_N)\Psi^\intercal\Psi =  \Psi^\intercal J_N\Psi = \Psi^\intercal\Psi b_\psi(J_N).$$
\end{proof}
This perhaps should not be surprising.
The orthogonality identity for binomial coefficients \eqref{orthogonal} implies that 
\begin{equation}\label{inv}
    \Psi_N\Lambda_N\Psi_N\Lambda_N=I_N.
\end{equation}
Therefore it follows that
$$S_N = \Lambda_N T_N^{-1}\Lambda_N.$$
Since the matrix $J_N$ commutes with $T_N$ (and hence $T_N^{-1}$), it is obvious that $S_N$ commutes with $b_\psi(J_N)=-\Lambda_NJ_N\Lambda_N + (N^2-1)I_N$.

\section{Spectra and the binomial transform}
The (signed) \vocab{binomial transform} of a sequence $\{a_n\}_{n=0}^\infty$ is the infinite sequence $\{s_n\}_{n=0}^\infty$ defined by
$$s_n = \sum_{k=0}^n\binom{n}{k}(-1)^k a_k.$$
Equivalently, this is the linear transformation, associated by the semi-infinite matrix product
$$\Psi\Lambda.$$
In this section, we use the truncated binomial transform $\Psi_N\Lambda_N$ to deduce interesting relationships between the spectra of $J_N$, $\Psi_N$ and $T_N$. 
%Furthermore,  the binomial transform of a sequence $\{v_j\}$ is the sequence whose $j$-th entry, for $j\in\mathbb N$, is defined by $$  \sum_{k=0}^j\binom{j}{k}(-1)^kv_k. $$ In other words, the binomial transform is the map defined, for any infinite vector $v$, by $$\vec v\mapsto \Psi\Lambda\vec v,$$  where $\Lambda$ is the semi-infinite diagonal matrix with entries $\Lambda_{jj}=(-1)^j$. If we restrict ourselves to finite vectors of size $N$, by abuse of notation, we have the  binomial transform defined by $$\vec v\mapsto \Psi_N\Lambda_N\vec v,$$  i.e., the binomial transform of a vector  $\vec v$, of size $N$, is the vector of size $N$ whose entries are $$(\Psi_N\Lambda_N\vec v)_j = \sum_{k=0}^j\binom{j}{k}(-1)^kv_k.$$
Furthermore,
we obtain an eigenbasis of the (finite) binomial transformation $\Psi_N\Lambda_N$ directly from the eigenvectors of $J_N$.

To begin, note that the definition of the generalized Fourier map implies that $$\pi(b_\psi(L)) = (\Psi^{-1}\pi(L)\Psi)^\intercal,$$
for any $L\in \mathcal F_x(\psi)$.
This means that
$$b_\psi(J_N) = \Psi_N^{-1}J_N\Psi_N.$$
Comparing this with the equation \eqref{bJ} from the previous section, we find
\begin{equation}\label{J_N}
    \Psi_N^{-1}J_N\Psi_N = -\Lambda_N J_N \Lambda_N + (N^2-1)I_N.
\end{equation}
This leads to the following theorem about the symmetry of the spectrum of $J_N$.
\begin{theorem}\label{T-reflection}
Let $\vec v$ be an eigenvector of $J_N$ with eigenvalue $\lambda$.
Then $\Psi_N\Lambda_N\vec v$ is an eigenvector of $J_N$ with eigenvalue $N^2-1-\lambda$.
In particular, the spectrum of $J_N$ is preserved under the reflection $\lambda\mapsto N^2-1-\lambda$.
\end{theorem}
\begin{proof}
This follows from the fact that $J_N$ and $-J_N+(N^2-1)I_N$ are similar after conjugation by $\Psi_N\Lambda_N$, as shown by Equation \eqref{J_N} combined with Equation \eqref{inv}.
\end{proof}
This theorem in particular says that the binomial transform % $\vec v\mapsto \Psi_N\Lambda_N\vec v$ with entries
%$$(\Psi_N\Lambda_N\vec v)_j = \sum_{k}\binom{j}{k}(-1)^kv_k$$
 acts as an involution that interchanges different eigenspaces of $J_N$.

One immediate consequence of this theorem is  that $J_N$ must have an eigenvector with eigenvalue $(N^2-1)/2$ when $N$ is odd.
\begin{corollary}
Suppose that $N$ is odd.
Then $(N^2-1)/2$ is an eigenvalue of $J_N$.
\end{corollary}
\begin{proof}
We can see this immediately from calculating the trace of $J_N$, since all but one of the eigenvalues of $J_N$ is paired with an eigenvalue of the opposite sign.
Alternatively, we can use the fact that the spectrum of $J$ is simple.
Then we write the eigenvalues of $J$ as an increasing sequence
$$\lambda_0 < \lambda_1 < \dots < \lambda_{N-1}.$$
The previous theorem implies that $$\lambda_i = N^2-1-\lambda_{N-1-i},\quad 0 \leq i < N-1.$$
In particular $\lambda_{(N-1)/2}=N^2-1-\lambda_{(N-1)/2}$.
\end{proof}

Since eigenvectors of $J_N$ are also eigenvectors of $T_N$, %comment: should we explain the simple spectrum of J_N?
this also implies a symmetry and relationship with the spectral properties of $T_N$.
This is captured by the next corollary.
\begin{corollary}
If $\vec v$ is an eigenvector of $T_N$ with eigenvalue $\lambda$, then its  binomial transform $\Psi \Lambda \vec v$ is an eigenvector of $T_N$ with eigenvalue $\lambda^{-1}$.
\end{corollary}
\begin{proof}
Since $T_N=\Psi\Psi^\intercal$, from Equation \eqref{inv} we have that 
$S_N=\Psi_N^\intercal\Psi_N = \Lambda_NT_N^{-1}\Lambda_N$. Hence,
$$T_N\Psi_N\Lambda_N\vec v = \Psi_N S_N\Lambda_N\vec v = \Psi_N\Lambda_NT_N^{-1}\vec v = \lambda^{-1}\Psi_N\Lambda_N\vec v.$$
\end{proof}
\begin{remark}
This also implies $T_N$ has an eigenvalue of $1$ when $N$ is odd, corresponding to the eigenvector of $J_N$ with eigenvalue $(N^2-1)/2$.
\end{remark}

The action of the binomial transformation on the spectra of $J$ provides an interesting way of obtaining an eigenbasis of the binomial transformation.
\begin{proposition}
If $\vec v$ is an eigenvector of $J_N$ with eigenvalue $\lambda\neq (N^2-1)/2$, then $\vec v \pm \Psi_N\Lambda_N\vec v$ is an eigenvector of $\Psi_N\Lambda_N$ with eigenvalue $\pm 1$.
In  the special case $\lambda=(N^2-1)/2$, $\vec v$ is already an eigenvector of $\Psi_N\Lambda_N$ with eigenvalue $1$.
\end{proposition}
\begin{proof}
Suppose that $\vec v$ is an eigenvector of $J_N$ with eigenvalue $\lambda\neq (N^2-1)/2$.
Then $\Psi_N\Lambda_N\vec v$ is an eigenvector with eigenvalue different from $\lambda$, so $\vec v\pm \Psi_N\Lambda_N\vec v\neq 0$.
Moreover, $(\Psi_N\Lambda_N)^2=I_N$ implies that
$$\Psi_N\Lambda_N(\vec v\pm \Psi_N\Lambda_N\vec v) = \pm (\vec v\pm \Psi_N\Lambda_N\vec v).$$
Thus $\vec v\pm \Psi_N\Lambda_N\vec v$ is an eigenvector of $\Psi_N\Lambda_N$ with eigenvalue $\pm 1$.

Lastly, we consider the special case that $N$ is odd and $\vec v$ is an eigenvector of $J_N$ with  $\lambda = (N^2-1)/2$.
Then $\Psi_N\Lambda_N \vec v$ is also an eigenvector of $J_N$ with eigenvalue $\lambda$.
Since the spectrum of $J_N$ is simple, it follows that
$$\Psi_N\Lambda_N\vec v = c\vec v$$
for some constant $c$.
The trace of $\Psi_N\Lambda_N$, is the sum of all of the eigenvalues, which come in pairs of $+1,-1$ pairs, except for $c$.
Therefore
$$c = \tr(\Psi_N\Lambda_N) = \sum_{k=0}^{N-1} (-1)^k = 1.$$
Thus $\vec v$ is an eigenvector of $\Psi_N\Lambda_N$ with eigenvalue $1$.
\end{proof}

When $N$ is odd, an eigenvector of $T_N$ with eigenvalue $1$, or equivalently, an eigenvector of $J_N$ with eigenvalue $(N^2-1)/2$ has rational entries and a simple recursive formula defined in terms of the entries of $J_N$.
Specifically, the entries $v_0,\dots,v_{N-1}$ of an eigenvector $\vec v$ satisfy $v_0=1$, $v_1=1/2$, and
\begin{equation}
v_{k+1} = \frac{\frac{1}{2}(N^2-1)v_k+k(N^2-2k^2-3k-2)v_k - k(N^2 - k^2)v_{k-1}}{(k+1)(N^2 - (k+1)^2)} ,
\end{equation}
for $0< k < N-1$.

\sectionnew{Numerical diagonalization of the Pascal matrix}\label{num}
As stated in the introduction, the problem of diagonalizing the $N\times N$ Pascal matrix $T_N$ is numerically unstable, even for relatively low values of $N$.
In contrast, diagonalization of the matrix $J_N$ is numerically well-behaved.
Furthermore, the theory of spectra of Jacobi matrices implies that $J_N$ has simple spectrum.
This means that an eigenbasis of $J_N$ must be an eigenbasis of $T_N$ as well.
This gives us a simple and natural workaround for the diagonalization of $T_N$ that is numerically robust.

The explicit formula for the vectors of the Pascal matrix or eigenvectors of $J_N$ is unknown.
Even so, standard eigenvector and eigenvalue routines are sufficient to approximate both.
In this section we use standard LAPACK routines \cite{lapack} to calculate the eigenvectors of $T_N$, both directly from $T_N$ and alternatively from $J_N$, and compare the resultant error.
This is implemented through the NumPy Python library \cite{numpy}.

To calculate the error in the computed eigenvalues, we need to calculate the eigenvectors more precisely as well.
We use a variable precision eigenvector and eigenvalue routines implemented from the mpmath Python library \cite{mpmath} to accurately calculate the eigenvectors to $200$ decimal places.
We use this high-accuracy calculation as our ``true" eigenvectors and take the difference between the true eigenvectors and the computed eigenvectors as the error.
In each case the eigenbasis is taken to be orthonormal.
Moreover, the spectra of both $T_N$ and $J_N$ are simple, so there is no ambiguity in the eigenvector comparisons.

\begin{figure}[htp]
\begin{tabular}{c|c|c|c}
eigenvalue of $T_N$ & eigenvalue of $J_N$ & error from $J_N$ & error from $T_N$\\\hline
$1.87658533e-08$ & $-2935.4$& $8.4094e-16$& $9.075e-6$\\
$ 1.16639323e-06$ & $-2319.7$& $7.1595e-16$& $9.3672e-6$\\
$3.31357255e-05$ & $-1762.4$& $5.055e-16$& $2.6413e-6$\\
$5.67427242e-04$ & $-1262.8$& $1.0187e-15$& $6.8088e-7$\\
$6.48778221e-03$ & $-821.21$& $8.7125e-16$& $1.6673e-7$\\
$5.15247212e-02$ & $-439.43$& $3.6927e-16$& $2.5406e-8$\\
$2.80569832e-01$ & $-124.69$& $3.307e-16$& $3.6577e-9$\\
$1.00000000e+0$ & $112.0$& $2.4291e-16$& $2.5178e-9$\\
$3.56417507e+0$ & $348.69$& $6.2037e-16$& $5.4666e-10$\\
$1.94081593e+01$ & $663.43$& $7.4917e-16$& $6.8609e-11$\\
$1.54135877e+02$ & $1045.2$& $7.0933e-16$& $1.4533e-11$\\
$1.76234048e+03$ & $1486.8$& $1.0204e-15$& $1.7079e-12$\\
$3.01789077e+04$ & $1986.4$& $1.166e-15$& $1.1491e-13$\\
$8.57343794e+05$ & $2543.7$& $8.01e-16$& $7.1326e-15$\\
$5.32882775e+07$ & $3159.4$& $1.3531e-16$& $2.7205e-16$
\end{tabular}
\caption{Comparison of the $\ell^2$-norm errors in eigenvectors of $T_N$, calculated from diagonalizing $T_N$ versus diagonalizing $J_N$ for $N=15$.}
\label{fig:eigenerror}
\end{figure}

The results in the case $N=15$ are summarized in the table in Figure \ref{fig:eigenerror}.
As we see, the eigenvectors computed from $J_N$ agree with the true eigenvector up to the full double-precision accuracy.
In contrast, the eigenvectors calculated from $T_N$ demonstrate significant numerical error, especially for the eigenvalues of $T_N$ which are clustered near zero.
As $N$ increases, the error in the eigenvectors calculated from $T_N$ continues to significantly grow.
By $N=20$, the errors for the eigenvalues near zero are on the same scale as the entries of the eigenvector itself, making some eigenvectors calculated from $T_N$ essentially random.
On the other hand, the precision in the calculation of eigenvalues and eigenvectors from $J_N$ remains essentially the same.

\section*{Acknowledgements}
\thanks{The research of W.R.C. has been supported by an AMS-Simons Research Enhancement Grant, and RSCA intramural grant 0359121 from CSUF; that of I.Z. was supported 
by Ministerio de Ciencia e Innovaci\'on PID2021-124332NB-C21, Universidad de Sevilla VI PPIT-US and Oberwolfach research in pairs program.}

%%%%%%%%%%%%%%%%%%%%%%%%%%%%%%%%%%%%%%%%%%%%%%%%%%%%%%%%%%%%%%%%%%%%%%%%%%%%%%
\end{document}